\documentclass[12pt]{amsart}
\usepackage{amssymb}
\usepackage{amsmath}
\usepackage{amsthm}
\usepackage{amsbsy}
\usepackage{psfrag}
\usepackage{graphics}
\usepackage{graphicx}
\theoremstyle{plain}
\newtheorem{theorem}{Theorem}[section]
\newtheorem{lemma}[theorem]{Lemma}
\newtheorem{proposition}[theorem]{Proposition}

\theoremstyle{remark}
\newtheorem{remark}[theorem]{Remark}

\begin{document}
\allowdisplaybreaks[4]
\numberwithin{figure}{section}
\numberwithin{table}{section}
 \numberwithin{equation}{section}
%
\title[$C^0$ IP  Method for Optimal Control Plate Problem]
 {A Frame Work for the Error Analysis of Discontinuous Finite Element Methods for Elliptic Optimal Control Problems and Applications to $C^0$ IP methods}

\author[S. Chowdhury]{Sudipto Chowdhury}

\address{Department of Mathematics, Indian Institute of Science, Bangalore - 560012}
\email{sudipto10@math.iisc.ernet.in}

\author[T. Gudi]{Thirupathi Gudi}

\address{Department of Mathematics, Indian Institute of Science, Bangalore - 560012}
\email{gudi@math.iisc.ernet.in}
\author[A. K. Nandakumaran]{A. K. Nandakumaran}

\address{Department of Mathematics, Indian Institute of Science, Bangalore - 560012}
\email{nands@math.iisc.ernet.in}

\date{}
\begin{abstract}
In this article, an abstract framework for the error analysis of
discontinuous Galerkin methods for control constrained optimal
control problems is developed. The analysis establishes the best
approximation result from  a priori analysis point of view and
delivers reliable and efficient a posteriori error estimators. The
results are applicable to a variety of problems just under the
minimal regularity possessed by the well-posed ness of the
problem. Subsequently, applications of $C^0$ interior penalty
methods for a boundary control problem as well as a distributed
control problem governed by the biharmonic equation subject to
simply supported boundary conditions are discussed through the
abstract analysis. Numerical experiments illustrate the
theoretical findings. Finally, we also discuss the variational
discontinuous discretization method (without discretizing the
control) and its corresponding error estimates.
\end{abstract}
\maketitle
\def \R{{{\Bbb R}}}
\def \P{{{\rm {\!}\cal P}}}
\def \Z{{{\rm {\!}\cal Z}}}
\def \Q{{{\rm {\!}\cal Q}}}
\allowdisplaybreaks
\def\d{\displaystyle}
\def\R{\mathbb{R}}
\def\cA{\mathcal{A}}
\def\cB{\mathcal{B}}
\def\p{\partial}
\def\O{\Omega}
\def\bbP{\mathbb{P}}
\def\bb{{\bf b}}
\def\cV{\mathcal{V}}
\def\cM{\mathcal{M}}
\def\cT{\mathcal{T}}
\def\cE{\mathcal{E}}
\def\bbE{\mathbb{E}}
\def\ssT{{\scriptscriptstyle T}}
\def\HT{{H^2(\O,\cT_h)}}
\def\mean#1{\left\{\hskip -5pt\left\{#1\right\}\hskip -5pt\right\}}
\def\jump#1{\left[\hskip -3.5pt\left[#1\right]\hskip -3.5pt\right]}
\def\smean#1{\{\hskip -3pt\{#1\}\hskip -3pt\}}
\def\sjump#1{[\hskip -1.5pt[#1]\hskip -1.5pt]}
\def\jumptwo{\jump{\frac{\p^2 u_h}{\p n^2}}}

\section{Introduction}\label{sec:Introduction}

The optimal control problems have been playing a very important
role in the modern scientific world. The numerical analysis for
these class of problems dates back to 1970's
\cite{Falk:1973:Control,Gevec:1979:Control}. There are many
landmark results on the finite element analysis of optimal control
problems. It is difficult to cite all the articles here but the
relevant work can be found in the references of some of the
articles that we discuss here. We refer to the monograph
\cite{trolzstch:2005:Book} for the theory of optimal control
problems and the aspects on the respective numerical algorithms.
Therein, the primal-dual active set strategy algorithm developed
in \cite{HIK:2003:PrimalDual} is also discussed in the context of
the optimal control problems. We refer to
\cite{meyerrosch:2004:Optiaml} for a super-convergence result for
a post-processed control for constrained control problems. A
variational discretization method is introduced in
\cite{Hinze:2005:Control} to derive optimal error estimates by
exploiting the relation between the control and the adjoint state.
For the Neumann boundary control problem with graded mesh
refinement refer to \cite{AJR:2012:Boundarycontrol} and for the
Dirichlet boundary control problems refer to
\cite{CR:2013:Dirich,Deckelnick:2009:Control,Gunzberger:1991:Stokes,MRV:2013:Dirich,Steinbach:2010:BEM}
and references there in. There has been also a lot of interest on
the state constrained control problems, for example refer to
\cite{Deckelnick:2007:Control,Ortner:2011:Control,Hinze:2011:Control}
and references therein. In the context of a posteriori error
analysis of control constrained problems, we refer to
\cite{Hoppe:2008:Control}. A general framework for a posteriori
error analysis of conforming finite element methods for optimal
control problems with constraints on controls is derived in
\cite{KRS:2014:AdaptiveControl} recently. The result therein is
obtained by the help of corresponding linear problems. In the
context of higher order problems, recently in
\cite{cao:2009:mixedBH,wollner:2012:MixedBH}, mixed finite element
methods have been proposed and analyzed for a distributed control
problem governed by the biharmonic equation subject to the
Dirichlet boundary conditions while a $C^0$ interior penalty
method is analyzed in \cite{Gudi:2014:Control} for the clamped
plate control problem.

There are not many results on the analysis of discontinuous
Galerkin (DG) methods for  optimal control problems, in particular
for higher order problems since DG methods are very attractive for
them. In this article, we develop an abstract framework for the
error analysis of discontinuous finite element methods applied to
control constrained optimal control problems. The outcome of the
result is a best approximation result for the method and a
reliable and efficient error estimator. It is important to note
that these best approximation results are key estimates in
establishing the optimality of adaptive finite element methods,
see for example \cite{Carsten:2013:CR,Hu:2012:Morley}. Also it is
worth noting that the standard error analysis of DG methods
require additional regularity which do not exist in several cases
for example in simply supported plates or mixed boundary value
problems e.g., see the discussions in
\cite{Gudi:2010:NewAnalysis,BNeilan:2010:C0IPSingular,GN:2011:Sixth,BGGS:2012:CHC0IP}.
Therefore the error  analysis of DG methods has to be treated
carefully. To this end, we introduce two auxiliary problems one is
dealt with a projection in a priori analysis and the other one is
based on a reconstruction in a posteriori analysis. Subsequently
Theorem \ref{thm:Err-q} and Theorem \ref{thm:Err-up-Apost} are
proved which play important role in the analysis. We believe that
the results in this article presents a framework for the error
analysis of discontinuous finite element methods for control
problems with limited regularity. Moreover the a posteriori error
estimator is useful in adaptive mesh refinement algorithms.

On the other hand, $C^0$ interior penalty methods became very
attractive in the recent past for approximating the solutions of
higher order problems
\cite{BGS:2010:AC0IP,BNeilan:2010:C0IPSingular,BSung:2005:DG4,BGGS:2012:CHC0IP,GN:2011:Sixth,EGHLMT:2002:DG3D}.
This is due to the fact that the conforming and mixed methods are
complicated and the nonconforming methods do not come in a natural
hierarchy. In this article, we propose and analyze a $C^0$
interior penalty method for optimal control problems (both
distributed and Neumann boundary control) governed by the
biharmonic equation subject to simply supported boundary
conditions. Note that the analysis of Dirichlet boundary control
problems in general is a subtle issue since the arguments for that
particular problem needs to be addressed using a very weak
formulation or an equivalent one, e.g., see
\cite{CR:2013:Dirich,MRV:2013:Dirich}. The analysis in this
article differs from the one in \cite{Gudi:2014:Control} and in
particular an abstract frame work for obtaining energy norm
estimates and in a posteriori error analysis. Also we analyze here
the boundary control problems. The variational discretization
method introduced in \cite{Hinze:2005:Control} is also discussed
in the context of discontinuous Galerkin methods using the
framework we developed here. Recently, it is shown in
\cite{Kanshat:2014:Stokes} that the $C^0$ interior penalty
solution of the biharmonic problem has connection to the
divergence-conforming solution of the Stokes problem. Therefore
our results will also be useful in the context of control problems
for Stokes equation.


The rest of the article is organized as follows. Section
\ref{sec:Model-Problem} introduces the model problems that are
under discussion. In Section \ref{sec:Abstract}, we set up the
abstract framework for the error analysis of discontinuous finite
element methods and derive therein some abstract error estimates
that form the basis for a priori and a posteriori error analysis.
In Section \ref{sec:Discrete-Problem}, we develop the discrete
setting and discuss the applications to the model problems
introduced in Section \ref{sec:Model-Problem}. In Section
\ref{sec:Numerics}, we present some numerical examples to
illustrate the theoretical results. In Section \ref{sec:VDM}, we
introduce the variational discretization method and sketch the
proofs for obtaining error estimates using the frame developed in
Section \ref{sec:Abstract}. Finally we conclude the article in
Section \ref{sec:conlusions}.

\section{Model Problems}\label{sec:Model-Problem}
The discussion will be on two model problems arising from the
optimal control of simply supported plate problem, one is on the
distributed control problem and the another is on the boundary
control problem. However the abstract analysis that we develop
later in the forthcoming section does not limit to only these
problems. In what follows subsequently, we introduce the common
data shared by the two model problems.

Let $\Omega\subset \R^2$ be a bounded domain with polygonal
boundary $\Gamma$. Assume that there is some $m\geq 1$ such that
the boundary $\Gamma$ is the union of some line segments
$\Gamma_i$'s $(1\leq i\leq m)$ whose interior in the induced
topology are pair-wise disjoint.  Let the admissible space
$V:=H^2(\Omega)\cap H^1_0(\Omega)$. Denote the $L_2(\Omega)$ and
$L_2(\Gamma)$ inner-products by $(\cdot,\cdot)$ and
$\langle\cdot,\cdot\rangle$, respectively. Let $f\in
H^{-1}(\Omega)$, $u_d\in L_2(\Omega)$ and  a real number
$\alpha>0$ are given. Define the bilinear form $a:V\times
V\rightarrow R$ by
\begin{align}\label{eq:def-a}
a(w,v)=(D^2 w, D^2v),
\end{align}
where $D^2w=[w_{x_ix_j}]_{1\leq i,j,\leq 2}$ is the standard
Hessian of $w$.

\begin{remark}
We may assume that the load function $f\in V^*$, the dual of $V$,
in that the case the numerical method will have to be modified.
The analysis in such cases can be handled as in
\cite{BCGG:2014:IMA}.
\end{remark}

\smallskip
\noindent {\bf Model Problem 1.} Define the quadratic functional
$J:V\times L_2(\Gamma)\rightarrow R$ by
\begin{align}\label{eq:def-J1}
J(w,p)=\frac{1}{2}\|w-u_d\|_{L_2(\Omega)}^2+\frac{\alpha}{2}\|p\|_{L_2(\Gamma)}^2,
\quad w\in V,\;p\in L_2(\Gamma).
\end{align}
\smallskip
\noindent For given $\underline q,\,\overline q \in \R\cup\{\pm
\infty\}$ with $\underline q<\overline q$, define the admissible
set of controls by
$$Q_b=\{p\in L_2(\Gamma):\underline q \leq p(x)
\leq \overline q\text{ for
 a.e. } x\in\Gamma\}.$$
\par
\noindent Consider the optimal control problem of finding $u\in V$
and $q\in Q_{b}$ such that
\begin{align}\label{eq:min-J1}
J(u,q)=\min_{w\in V,p\in Q_{b}}J(w,p),
\end{align}
subject to the condition that $w\in V$ satisfies
\begin{align}
a(w,v)=f(v)+\langle p,\partial v/\partial n\rangle \quad \forall
v\in V. \label{eq:MP-w1}
\end{align}
Note that the optimal solution $(u,q)\in V\times Q_{b}$, whenever
exists, satisfies
\begin{align}
a(u,v)=f(v)+\langle q,\partial v/\partial n\rangle \quad \forall
v\in V.\quad\label{eq:MP-u1}
\end{align}

\smallskip

In order to establish the existence of a solution to
\eqref{eq:min-J1}, note that the model problem \eqref{eq:MP-w1} has a unique solution $w\in V$
for given $p\in L_2(\Gamma)$.  Define this correspondence as $Sp=w$. From the stability estimates of the
solution $w$, it is easy to check that $S: L_2(\Gamma)\rightarrow
L_2(\Omega)$ defines a continuous linear operator. Using the
operator $S$, the minimization problem \eqref{eq:min-J1} can be
written in the reduced form of finding $q\in L_2(\Gamma)$ such
that
\begin{align}\label{eq:min-j1}
j(q)=\min_{p\in Q_{b}}j(p),
\end{align}
where
\begin{align}\label{eq:def-j1}
j(p)=\frac{1}{2}\|Sp-u_d\|_{L_2(\Omega)}^2+\frac{\alpha}{2}\|p\|_{L_2(\Gamma)}^2.
\end{align}
\par
\noindent
Using the theory of elliptic optimal control problems
\cite{trolzstch:2005:Book}, the following proposition on the existence and uniqueness of the solution can be proved and the optimality condition can be derived.

\begin{proposition}
The control problem $\eqref{eq:min-j1}$ has a
unique solution $q$ and the corresponding solution $u=Sq$ of
$\eqref{eq:MP-u1}$. Furthermore, by introducing the adjoint state $\phi\in V$ such
that
\begin{align}
a(v, \phi)=( u-u_d,v)\quad \forall v\in V,\quad\label{eq:MP-p1}
\end{align}
the optimality condition that $j\,'(q)(p-q)\geq 0$, $\forall p\in Q_b$, can be
expressed as
\begin{align}\label{eq:MP-q1}
\langle \partial \phi/\partial n+\alpha  q,p- q\rangle \geq 0\quad
\forall p\in Q_{b}.
\end{align}
\end{proposition}

\medskip \noindent {\bf Model Problem 2.} In this example, we
consider  the model of a distributed control problem. For
this, define the quadratic functional $J:V\times
L_2(\Omega)\rightarrow R$ by
\begin{align}\label{eq:def-J2}
J(w,p)=\frac{1}{2}\|w-u_d\|_{L_2(\Omega)}^2+\frac{\alpha}{2}\|p\|_{L_2(\Omega)}^2,
\quad w\in V,\;p\in L_2(\Omega).
\end{align}
\smallskip
\noindent Let $\underline q,\,\overline q \in \R$ with $\underline
q<\overline q$ be given.  Define $Q_d=\{p\in
L_2(\Omega):\underline q \leq p(x) \leq \overline q\text{ for
 a.e.} x\in\Omega\}$.
 The distributed control problem consists of finding $u\in V$
and $q\in Q_{d}$ such that
\begin{align}\label{eq:min-J2}
J(u,q)=\min_{w\in V,p\in Q_{d}}J(w,p),
\end{align}
where $w\in V$ satisfies
\begin{align}
a(w,v)=f(v)+(p, v)\quad \forall v\in V. \label{eq:MP-w2}
\end{align}
It is clear that whenever it exists the optimal solution $(u,q)\in
V\times Q_{d}$ satisfies
\begin{align}
a(u,v)=f(v)+(q,v) \quad \forall v\in V.\quad\label{eq:MP-u2}
\end{align}

\smallskip
Note that the model problem \eqref{eq:MP-w2} has a unique solution
$w\in H^{2}(\Omega)$ for given $p\in L_2(\Omega)$. Setting this
correspondence as $Sp=w$ and using the stability estimates of $w$,
it is obvious that $S: L_2(\Omega)\rightarrow L_2(\Omega)$ defines
a continuous linear operator. Then, the minimization problem
\eqref{eq:min-J2} is reduced to find $q\in L_2(\Omega)$ such that
\begin{align}\label{eq:min-j2}
j(q)=\min_{p\in Q_{d}}j(p),
\end{align}
where
\begin{align}\label{eq:def-j2}
j(p)=\frac{1}{2}\|Sp-u_d\|_{L_2(\Omega)}^2+\frac{\alpha}{2}\|p\|_{L_2(\Omega)}^2.
\end{align}
Again the theory of elliptic optimal control problems
\cite{trolzstch:2005:Book} implies that the problem
\eqref{eq:min-j2} has a unique solution $q$. The corresponding
solution of \eqref{eq:MP-u2} is denoted by $ u$. Moreover as in
the earlier case, there exists an adjoint state $\phi\in V$ such
that
\begin{align}
a(v, \phi)=( u-u_d,v)\quad \forall v\in V,\quad\label{eq:MP-p2}
\end{align}
and
\begin{align}\label{eq:MP-q2}
(\phi+\alpha  q,p- q ) \geq 0\quad \forall p\in Q_{d}.
\end{align}

\begin{remark}
The quadratic functionals $J$ in \eqref{eq:def-J1} or
\eqref{eq:min-J2} may consists of $\|w-u_d\|_{H^k(\Omega)}$ for
$k=1$ or $2$. The analysis in the forthcoming section can be
easily extended to these cases as well.
\end{remark}

\section{Abstract Setting and Analysis}\label{sec:Abstract}
In this section, we develop an abstract framework for the error
analysis of discontinuous and nonconforming methods for
approximating the solutions of optimal control problems with
either boundary control or the distributed control. All the vector spaces introduced below
are assumed to be real.

\par
Let $V$ be a Sobolev-Hilbert space with the norm $\|\cdot\|_V$ and
with dual denoted by $V'$. The space $V$ will be an admissible
space for state and adjoint state variables. Let $W$ be a Hilbert
space such that $V\subset W\subset V'$ (Gelfand triplet) and the
inclusion is continuous. The inner product and norm on $W$ is
denoted by $(\cdot,\cdot)$ and $\|\cdot\|_W$, respectively.
Let $Q$ be an Hilbert space that will be used for seeking the control variable.
The norm and the inner-product on $Q$ will be denoted by
$\|\cdot\|_Q$  and $\langle\cdot,\cdot\rangle$ respectively. Let
$B:V\rightarrow Q$ be a linear and continuous operator. Let
$Q_{ad}\subset Q$ be a nonempty, closed and convex subset.

\par
Assume that $(u,\phi,q)\in V\times V\times Q_{ad}$ solve the
system
\begin{align}
a(u,v)&=f(v)+\langle q,Bv\rangle \quad \forall v\in V,\quad\label{eq:AMP-u}\\
a(v,\phi)&=(u-u_d,v)\quad \forall v\in V,\quad\label{eq:AMP-p}\\
\langle B\phi+\alpha q, p- q\rangle &\geq 0\quad \forall p\in
Q_{ad}, \label{eq:AMP-q}
\end{align}
where $f \in V'$, $u_d\in W$, $\alpha>0$ are given and $a:V\times V
\rightarrow R$ is a continuous and elliptic bilinear form in
the sense that there exist positive constants $C$ and $c$ such
that
\begin{align*}
|a(u,v)|&\leq C \|u\|_V\;\|v\|_V\quad \forall u,\,v\in V,\\
a(v,v)&\geq c\|v\|_V^2\quad \forall v\in V.
\end{align*}
\smallskip
\noindent Next we introduce the corresponding discrete setting.
Let $V_h\subset W$ be a finite dimensional (finite element)
subspace and there is a norm $\|\cdot\|_h$ on $V_h+V$ such that
$\|v\|_h=\|v\|_V$ for all $v\in V$. Let $a_h:V_h\times
V_h\rightarrow R$ be a continuous and elliptic bilinear form, i.e,
there exist positive constants $\tilde C$ and $\tilde c$ such that
\begin{align*}
|a_h(u_h,v_h)|&\leq \tilde C \|u_h\|_h\;\|v\|_h\quad \forall u_h,\,v_h\in V_h,\\
a_h(v_h,v_h)&\geq \tilde c \|v_h\|_h^2\quad \forall v_h\in V_h.
\end{align*}
Similarly assume that $Q_h\subset Q$ be a finite dimensional
(finite element) subspace and $Q_{ad}^h\subset Q_{ad}$ be a
nonempty, closed and convex subset of $Q_h$. Further assume that
$f\in V_h'$.

\par
Suppose that the discrete variables $(u_h,\phi_h,q_h)\in V_h\times
V_h\times Q_{ad}^h$ solve the system
\begin{align}
a_h(u_h,v_h)&=f(v_h)+\langle q_h,B_hv_h\rangle \quad \forall v_h\in V_h,\quad\label{eq:ADP-u}\\
a_h(v_h,\phi_h)&=(u_h-u_d,v_h)\quad \forall v_h\in V_h,\quad\label{eq:ADP-p}\\
\langle B_h\phi_h+\alpha q_h, p_h- q_h\rangle &\geq 0\quad \forall
p_h\in Q_{ad}^h, \label{eq:ADP-q}
\end{align}
where $B_h:V_h\rightarrow Q_h$ is discrete counterpart of $B$ such
that $B_hv=Bv$ for all $v\in V$.

Throughout this section, we assume that the following hold true:
\par
\smallskip
\noindent {\bf Assumption (P-T):} There hold
\begin{align}
\|v\|_W &\leq C \|v\|_h \quad \forall v\in V+V_h,\label{eq:Poincare-type}\\
\|B_h(v-v_h)\|_Q &\leq  C \|v-v_h\|_h\quad \text{ for } v\in
V,\;v_h\in V_h.\label{eq:trace-QV}
\end{align}
As it will be seen in subsequent sections that
\eqref{eq:Poincare-type} corresponds to a Poincar\'e type
inequality and \eqref{eq:trace-QV} corresponds to a trace
inequality on broken Sobolev spaces.

We need the $Q$-projection defined by the following: For given
$q\in Q$, let $\Pi_h q\in Q_h$ be the solution of
\begin{align}
\langle \Pi_h q-q, p_h\rangle=0 \quad \forall p_h\in
Q_h.\quad\label{eq:Piq}
\end{align}
\par
\smallskip
\noindent {\bf Assumption (Q):} Assume that $\Pi_hq\in Q_{ad}^h$
whenever $q\in Q_{ad}$.

\par
\smallskip
We turn to derive some abstract a priori error analysis. To this
end, we introduce some projections as follows: Let $P_hu\in V_h$
and $P_h\phi\in V_h$ solves
\begin{align}
a_h(P_hu,v_h)&=f(v_h)+\langle q,B_hv_h\rangle \quad \forall v_h\in V_h,\quad\label{eq:Phu}\\
a_h(v_h, P_h\phi)&=(u-u_d,v_h)\quad \forall v_h\in
V_h,\quad\label{eq:Php}
\end{align}
respectively.

\par
\smallskip
\noindent The following lemma is a key in the error analysis.

\begin{lemma}\label{lem:Err-q-abstract} There hold
\begin{align*}
\langle B_h(\phi_h-P_h\phi), q-q_h\rangle &\geq \alpha
\|q-q_h\|_Q^2+\langle B_h\phi_h+\alpha q_h, q-p_h\rangle \\
&\quad+\langle B_h(\phi-P_h\phi), q-q_h\rangle \quad \forall
p_h\in Q_{ad}^h.
\end{align*}
\end{lemma}
\begin{proof}
Since $Q_{ad}^h\subset Q_{ad}$ and $B_h=B$ on $V$, we find from
\eqref{eq:AMP-q} and \eqref{eq:ADP-q} that
\begin{align*}
\langle B_h\phi_h+\alpha q_h, q-q_h\rangle
&\geq \langle B_h\phi_h+\alpha q_h, q-p_h\rangle  \quad \forall p_h\in Q_{ad}^h.\\
-\langle B_h\phi+\alpha q, q-q_h\rangle &\geq 0.
\end{align*}
We find by adding the above two inequalities that
\begin{align*}
\langle B_h(\phi_h-\phi)+\alpha (q_h-q), q-q_h\rangle &\geq
\langle B_h\phi_h+\alpha q_h, q-p_h\rangle \quad \forall p_h\in
Q_{ad}^h,
\end{align*}
which implies
\begin{align*}
\langle B_h(\phi_h-P_h\phi), q-q_h\rangle &\geq \alpha
\|q-q_h\|_Q^2+\langle B_h\phi_h+\alpha q_h, q-p_h\rangle \\
&\quad+\langle B_h(\phi-P_h\phi), q-q_h\rangle \quad \forall
p_h\in Q_{ad}^h.
\end{align*}
This completes the proof.
\end{proof}

The following theorem derives an abstract a priori error estimate
for the control.

\begin{theorem}\label{thm:Err-q}
There hold
\begin{align*}
\| q- q_h\|_Q^2+\|u-u_h\|_W^2 &\leq C\left(\|B\phi-\Pi_h(B
\phi)\|_Q^2+\|q-\Pi_hq\|_Q^2+\|\phi-P_h\phi\|_h^2\right)\\&\quad
+C\|u-P_hu\|_W^2.
\end{align*}
\end{theorem}
\begin{proof}
From \eqref{eq:ADP-u}-\eqref{eq:ADP-p} and the definition of
$P_h$, we have
\begin{align}
a_h(P_hu-u_h,v_h)&= \langle q-q_h,B_hv_h\rangle\quad \forall v_h\in V_h,\label{eq:E11}\\
a_h(v_h, P_h\phi-\phi_h)&=(u-u_h,v_h)\quad \forall v_h\in
V_h.\label{eq:E12}
\end{align}
Take $v_h=P_hu-u_h$ in \eqref{eq:E12}, $v_h=P_h\phi-\phi_h$ in
\eqref{eq:E11} and subtract the resulting equations to find
\begin{align*}
\langle q-q_h, B_h(P_h\phi-\phi_h)\rangle-(u-u_h,P_hu-u_h)=0.
\end{align*}
This implies
\begin{align*}
\langle
q-q_h,B_h(\phi_h-P_h\phi)\rangle+\|P_hu-u_h\|_W^2=(u-P_hu,u_h-P_hu).
\end{align*}
Using Lemma \ref{lem:Err-q-abstract} with $p_h=\Pi_hq$, we find
that
\begin{align}
\alpha \|q-q_h\|_Q^2+\|P_hu-u_h\|_W^2 &\leq -\langle
B_h\phi_h+\alpha q_h, q-p_h\rangle
-\langle B_h(\phi-P_h\phi), q-q_h\rangle \notag \\
&\quad + (u-P_hu,u_h-P_hu)\notag \\
&=-\langle B_h\phi+\alpha q, q-p_h\rangle
-\langle B_h(\phi-P_h\phi), q-q_h\rangle \notag\\
&\quad  -\langle B_h(\phi_h-\phi)+\alpha
(q_h-q), q-p_h\rangle \notag\\
&\quad + (u-P_hu,u_h-P_hu) \notag\\
&=-\langle B\phi-\Pi_h\left(B
\phi\right)+\alpha (q-p_h), q-p_h\rangle\notag\\
&\quad-\langle B_h(\phi-P_h\phi), q-q_h\rangle + (u-P_hu,u_h-P_hu) \notag\\
&\quad  -\langle B_h(\phi_h-\phi)+\alpha (q_h-q), q-p_h\rangle\notag\\
&=-\langle B\phi-\Pi_h\left(B
\phi\right)+\alpha (q-p_h), q-p_h\rangle\notag\\
&\quad-\langle B_h(\phi-P_h\phi), q-q_h\rangle + (u-P_hu,u_h-P_hu) \notag\\
&\quad  -\langle B_h(\phi_h-\phi)+\alpha (p_h-q), q-p_h\rangle\notag\\
&\leq C\left(\|B\phi-\Pi_h(B
\phi)\|_Q^2+\|q-p_h\|_Q^2+\|B_h(\phi-P_h\phi)\|_Q^2\right)\notag\\
&\quad+
C\left(\|B_h(\phi-\phi_h)\|_Q\;\|q-p_h\|_Q+\|u-P_hu\|_W^2\right)\notag\\
&\quad +
\frac{1}{2}\|u_h-P_hu\|_W^2+\frac{\alpha}{2}\|q-q_h\|_Q^2.\label{eq:Eq1}
\end{align}
From the error equation \eqref{eq:E12}, we have
\begin{align*}
a_h(P_h\phi_h-\phi_h,P_h\phi-\phi_h)=(u-u_h,P_h\phi-\phi_h)\leq C
\|u-u_h\|_W \|P_h\phi-\phi_h\|_W.
\end{align*}
By the assumption \eqref{eq:Poincare-type} and the ellipticity of
$a_h$, we find
\begin{align}
\|P_h\phi-\phi_h\|_h \leq C\|u-u_h\|_W.\label{eq:Eq2}
\end{align}
Now using the assumption \eqref{eq:trace-QV} and \eqref{eq:Eq2},
we find
\begin{align*}
\|B_h(\phi-\phi_h)\|_Q &\leq C \|\phi-\phi_h\|_h \leq C \|\phi-P_h\phi_h\|_h+ C \|P_h\phi-\phi_h\|_h\\
& \leq C \|\phi-P_h\phi_h\|_h+ C \|u-u_h\|_W\\
& \leq C \left(\|\phi-P_h\phi\|_h+\|u-P_hu\|_W+\|P_hu-u_h\|_W\right).
\end{align*}
Using this estimate in \eqref{eq:Eq1}, we complete the proof.
\end{proof}

We now derive the error estimates for the state and the adjoint
state variables.

\begin{theorem}\label{thm:Err-u} There hold
\begin{align*}
\|\phi-\phi_h\|_h & \leq C\left( \|\phi-P_h\phi\|_h + \|B\phi-\Pi_h(B
\phi)\|_Q+\|q-\Pi_hq\|_Q+\|u-P_hu\|_W\right),\\
\|u-u_h\|_h & \leq  C\left( \|\phi-P_h\phi\|_h + \|B\phi-\Pi_h(B
\phi)\|_Q+\|q-\Pi_hq\|_Q+\|u-P_hu\|_h\right).
\end{align*}
\end{theorem}

\begin{proof}
The estimate in \eqref{eq:Eq2} together with the estimate in
Theorem \ref{thm:Err-q} and the triangle inequality imply
\begin{align*}
\|\phi-\phi_h\|_h&\leq \|\phi-P_h\phi\|_h+\|P_h\phi-\phi_h\|\\
&\leq \|\phi-P_h\phi\|_h+ C\|u-u_h\|_W\\
&\leq C\left( \|\phi-P_h\phi\|_h + \|B\phi-\Pi_h(B
\phi)\|_Q+\|q-\Pi_hq\|_Q+\|u-P_hu\|_W\right).
\end{align*}
The error equation \eqref{eq:E11} and the assumption \eqref{eq:trace-QV} imply
\begin{align*}
a_h(P_hu-u_h,P_hu-u_h)&=\langle q-q_h,B_h(P_hu-u_h)\rangle\\
& =\langle q-q_h,B_h(P_hu-u)\rangle+\langle q-q_h, B_h(u-u_h)\rangle\\
&\leq C\left(\|u-P_hu\|_h+\|u-u_h\|_h\right)\|q-q_h\|_Q.
\end{align*}
The rest of the proof follows from Theorem \ref{thm:Err-q}.
\end{proof}

Next we will develop an abstract setting for a posteriori error
control. To this end, define the reconstructions  $Ru \in V$ and
$R\phi \in V$ by
\begin{align}
a(Ru,v)&=f(v)+\langle q_h,Bv\rangle \quad \forall v\in V,\label{eq:Ru}\\
a(v, R\phi)&=(u-u_d,v)\quad \forall v\in V.\label{eq:Rp}
\end{align}
From the above definitions and \eqref{eq:AMP-u}-\eqref{eq:AMP-p},
we have
\begin{align}
a(u-Ru,v)&=\langle q-q_h,Bv\rangle \quad \forall v\in V,\label{eq:ERu}\\
a(v,\phi-R\phi)&=(u-u_h,v)\quad \forall v\in V.\label{eq:ERp}
\end{align}

\par
\noindent The following lemma will be useful in the subsequent a
posteriori error analysis:

\begin{lemma}\label{lem:Err-q-Apost} There hold
\begin{align*}
\langle B(R\phi-\phi), q-q_h\rangle &\geq \alpha
\|q-q_h\|_Q^2+\langle B_h\phi_h+\alpha q_h, q-p_h\rangle \\
&\quad+\langle B_h(R\phi-\phi_h), q-q_h\rangle \quad \forall
p_h\in Q_{ad}^h.
\end{align*}
\end{lemma}
\begin{proof}
Using the assumptions  $Q_{ad}^h\subset Q_{ad}$, $B_h=B$ on $V$,
and the inequalities \eqref{eq:AMP-q} and \eqref{eq:ADP-q}, we
find
\begin{align*}
\langle B_h\phi_h+\alpha q_h, q-q_h\rangle
&\geq \langle B_h\phi_h+\alpha q_h, q-p_h\rangle  \quad \forall p_h\in Q_{ad}^h.\\
-\langle B_h\phi+\alpha q, q-q_h\rangle &\geq 0.
\end{align*}
Add the above two inequalities and find
\begin{align*}
\langle B_h(\phi_h-\phi)+\alpha (q_h-q), q-q_h\rangle &\geq
\langle B_h\phi_h+\alpha q_h, q-p_h\rangle \quad \forall p_h\in
Q_{ad}^h.
\end{align*}
This trivially implies
\begin{align*}
\langle B(R\phi-\phi), q-q_h\rangle &\geq \alpha
\|q-q_h\|_Q^2+\langle B_h\phi_h+\alpha q_h, q-p_h\rangle \\
&\quad+\langle B_h(R\phi-\phi_h), q-q_h\rangle \quad \forall
p_h\in Q_{ad}^h.
\end{align*}
Hence the proof.
\end{proof}

The first result that will be useful in a posteriori error
estimates for the control is the following:

\begin{theorem}\label{thm:Err-up-Apost} There hold
\begin{align*}
\|q-q_h\|_Q+\|u-Ru\|_W &\leq
C\left(\|Ru-u_h\|_W+\|R\phi-\phi_h\|_h\right)\\
&\quad +C\|(B_h\phi_h+\alpha q_h)-\Pi_h(B_h\phi_h+\alpha q_h)\|_Q.
\end{align*}
\end{theorem}
\begin{proof}
Taking $v=u-Ru$ in \eqref{eq:ERp} and $v=\phi-R\phi$ in \eqref{eq:ERu} and then subtracting the resulting equations,
\begin{align*}
\langle q-q_h,B(\phi-R\phi)\rangle-(u-u_h,u-Ru)=0.
\end{align*}
It trivially implies
\begin{align*}
\langle q-q_h,B(R\phi-\phi)\rangle+\|u-Ru\|_W^2=-(Ru-u_h,u-Ru).
\end{align*}
Using the estimate in Lemma \ref{lem:Err-q-Apost} in the above
equation, we find
\begin{align*}
\alpha \|q-q_h\|_Q^2 + \|u-Ru\|_W^2 &\leq  -(Ru-u_h,u-Ru)-\langle B_h\phi_h+\alpha q_h, q-p_h\rangle \\
&\quad-\langle B_h(R\phi-\phi_h), q-q_h\rangle \quad \forall
p_h\in Q_{ad}^h.
\end{align*}
Let $p_h=\Pi_hq \in Q^h_{ad}$. Then
\begin{align*}
\langle B_h\phi_h+\alpha q_h, q-\Pi_hq\rangle&=\langle
(B_h\phi_h+\alpha q_h)-\Pi_h(B_h\phi_h+\alpha q_h),\;
q-\Pi_hq\rangle\\
&=\langle (B_h\phi_h+\alpha q_h)-\Pi_h(B_h\phi_h+\alpha q_h),\;
q-q_h\rangle.
\end{align*}
The proof then follows from  the Cauchy-Schwarz inequality and the
assumption \eqref{eq:trace-QV}.
\end{proof}

Next, the result that will be useful in the a posteriori error
analysis of the state and the adjoint states is derived below.

\begin{theorem}\label{thm:Apost-up}
There hold
\begin{align*}
\|u-u_h\|_h+\|\phi-\phi_h\|_h&\leq C\left(\|Ru-u_h\|_h+\|R\phi-\phi_h\|_h\right)\\
&\quad +C\|(B_h\phi_h+\alpha q_h)-\Pi_h(B_h\phi_h+\alpha q_h)\|_Q.
\end{align*}
\end{theorem}
\begin{proof}
By the triangle inequality,
\begin{align*}
\|u-u_h\|_h\leq \|u-Ru\|_h+\|Ru-u_h\|_h.
\end{align*}
Taking $v=u-Ru$ in \eqref{eq:ERu} and since
$\|\cdot\|_h=\|\cdot\|_V$ on $V$, we find by using the continuity
of the operator $B$ that
\begin{align*}
\|u-Ru\|_V \leq C \|q-q_h\|_Q.
\end{align*}
The bound for $u-u_h$ then follows by using Theorem
\ref{thm:Err-up-Apost}. Similarly by the triangle inequality
\begin{align*}
\|\phi-\phi_h\|_h\leq \|\phi-R\phi\|_h+\|R\phi-\phi_h\|_h.
\end{align*}
Taking $v=\phi-R\phi$ in \eqref{eq:ERp} and again since
$\|\cdot\|_h=\|\cdot\|_V$ on $V$, we find by using the continuous
imbedding of $V$ in $W$ that
\begin{align*}
\|\phi-R\phi\|_V \leq C \|u-u_h\|_W.
\end{align*}
The rest of the proof follows from the assumption
\eqref{eq:Poincare-type} and the estimate for $\|u-u_h\|_h$.
\end{proof}

\section{A Specific and Discrete Setting}\label{sec:Discrete-Problem}
\subsection{Notations}
Denote the norm and semi-norm on $H^k(D)$ ($k\geq 0$) for any open domain $D\subset \R^s$ $(s\geq 1)$ by
$\|v\|_{k,D}$ and $|v|_{k,D}$. Note that the semi-norm $|\cdot|_{2,\Omega}$ defines a norm on $V=H^2(\Omega)\cap H^1_0(\Omega)$
which is equivalent to $\|\cdot\|_{2,\Omega}$.
Let $\cT_h$ be a regular simplicial subdivision of $\Omega$.
Denote the set of all interior edges/faces of $\cT_h$ by
 $\cE_h^i$, the set of boundary edges/faces by $\cE_h^b$, and
 define $\cE_h=\cE_h^i\cup\cE_h^b$.
 Let $h_T$=diam$ (T)$ and $h=\max\{ h_T : T\in\cT_h\}$.
 The diameter of any edge/face $e\in\cE_h$ will be denoted by $h_e$.
 We define the Sobolev space $H^s(\O,\cT_h)$ associated with the subdivision $\cT_h$
 as follows:
\begin{equation*}
  H^s(\O,\cT_h)=\{v\in L_2(\Omega) :\,v|_{T}\in H^s(T)\quad\forall\,~T\in \cT_h\}.
\end{equation*}
 The discontinuous finite element space is
\begin{align}
V_h=\{v\in  H^1_0(\Omega) : \, v|_{T}\in \bbP_2(T)\quad\forall\,~T\in\cT_h\},\label{eq:Vh-def}
\end{align}
 where $\bbP_2(D)$ is the space of polynomials of degree less than or equal to $2$
 restricted to the set $D$.
 It is clear that $V_h\subset  H^1_0(\Omega)\cap H^s(\O,\cT_h)$ for any positive integer $s$.
\par
 For any $e\in\cE_h^i$, there are two elements $T_+$ and $T_-$ such that
 $e=\partial T_+\cap\partial T_-$.
 Let $n_-$ be the unit normal of $e$ pointing from $T_-$ to $T_+$.
 For any $v\in H^2(\Omega,\cT_h)$, we define the jump of the normal
 derivative of $v$ on $e$ by
\begin{equation*}
 \sjump{\nabla v} =
 \left. \nabla v_+\right|_{e}\cdot n_+ +
 \left. \nabla v_-\right|_{e}\cdot n_-
\end{equation*}
 where $v_\pm=v\big|_{T_\pm}$.
 For any $v\in H^3(\O,\cT_h)$,
 we define the mean and jump of the second order normal derivative of $v$ across $e$ by
\begin{equation*}
 \smean{\partial ^2 v/\partial n^2} =
 \frac{1}{2}\left(\partial^2 v_{+}/\partial n^2
 +\partial^2 v_{-}/\partial n^2\right),
\end{equation*}
and
\begin{equation*}
 \sjump{\partial ^2 v/\partial n^2} =
 \left(\partial^2 v_{+}/\partial n^2
 -\partial^2 v_{-}/\partial n^2\right),
\end{equation*}
respectively, where $n$ is either $n_+$ or $n_-$ (the sign of $n$ will not change the above quantities).

\par
For notational convenience, we also define jump and average on the
boundary edges. For any $e\in \cE_h^b$, there is a element
$T\in\cT_h$ such that
 $e=\partial T\cap \partial\Omega$.
 Let $n_e$ be the unit normal of $e$ that points outside $T$.
  For any $v\in H^2(T)$, we set on $e$
\begin{equation*}
 \sjump{\nabla v}= \nabla v \cdot n_e,
\end{equation*}
 and for any $v\in H^3(T)$, we set
\begin{equation*}
 \smean{\partial ^2 v/\partial n^2}=\partial ^2 v/\partial n^2.
\end{equation*}

We require the following trace inequality
\cite{Grisvard:1985:Singularities}: {\lemma \label{lem:trace}
There holds for $v\in H^2(\Omega)\cap H^1_0(\Omega)$ that
\begin{eqnarray*}
\|\partial v/\partial n\|_{L_2(\Gamma_i)} \leq C \|v\|_{2,\Omega}\quad \forall 1\leq i\leq m.
\end{eqnarray*}}
We also use the following inverse inequality on $V_h$
\cite{BScott:2008:FEM,Ciarlet:1978:FEM}:
{\lemma\label{lem:inverse} For $v_h\in V_h$, there holds
\begin{align*}
\|v_h\|_{L_2(e)} & \leq C h_e^{-1/2}
\|v_h\|_{L_2(T)}\,\,\forall\,T\in \cT_h,
\end{align*}
where $e$ is an edge of $T$.}\hspace{2cm}\endproof

\par
\smallskip
\noindent
{\bf Enriching Map}. Let $V_c\subset H^2(\Omega)\cap
H^1_0(\Omega)$ be the Hsieh-Clough-Toucher $C^1$ finite element
space associated with the triangulation $\cT_h$ (see
\cite{BScott:2008:FEM,Ciarlet:1978:FEM,BNeilan:2010:C0IPSingular}).
In the error analysis of discontinuous Galerkin methods, we use an
enriching map $E_h:V_h \rightarrow V_c$ that plays an important
role. As it is done in \cite{BNeilan:2010:C0IPSingular}, we define
$E_h:V_h \rightarrow V_c$ as follows: Let N be any degree of
freedom of $V_c$ i.e., N is either the evaluation of a shape
function or its first order derivatives at any vertex   or the
evaluation of the normal derivative of shape function at the
midpoint of any edge in $\cE_h$. Then, for any $v \in V_h$,
\begin{equation*}
N(E_hv_h)=\frac{1}{|\cT_N|}\sum_{T \in \cT_N}N(v_T).
\end{equation*}
where $\cT_N$ is the set of triangles sharing the degree of
freedom N and $|\cT_N|$ denotes the cardinality of $\cT_N$.

\par
\noindent The following Lemma states the approximation properties
satisfied by the map $E_h$ \cite{BNeilan:2010:C0IPSingular}:

\begin{lemma}\label{lem:EnrichApprx} Let $v \in V_h$.
It holds that
\begin{align*}
\sum_{T \in
\cT_h}\left(h_T^{-4}\|E_hv-v\|^2_{0,T}+h_T^{-2}\|\nabla(E_hv-v)\|^2_{0,T}\right)&\leq
C\Big(\sum_{e \in \cE_h^i}\frac{1}{h_e}\Big\|\jump{\frac{\p v}{\p
n}}\Big\|^2_{0,e}\Big) \;\;\forall \; v \in V_h,
\end{align*}
and
\begin{align*}
\sum_{T \in \cT_h}|E_hv-v|^2_{H^2(T)}\leq C\Big(\sum_{e \in
\cE_h^i}\frac{1}{h_e}\Big\|\jump{\frac{\p v}{\p
n}}\Big\|^2_{0,e}\Big)\;\; \forall \; v \in V_h.
\end{align*}
\end{lemma}

\smallskip
\noindent Following \cite{BNeilan:2010:C0IPSingular}, the bilinear form for the numerical method is defined by
\begin{align}
a_h(w,v)&=\sum_{T\in\cT_h}\int_T D^2 w : D^2 v \,dx - \sum_{e\in \cE_h^i}\int_e\smean{\partial^2 w/\partial n^2}\sjump{\nabla v}\,ds \notag\\
& \quad- \sum_{e\in \cE_h^i}\int_e\smean{\partial^2 v/\partial n^2}\sjump{\nabla w}\,ds
 +\sum_{e\in \cE_h^i}\int_e\frac{\eta}{h_e}\sjump{\nabla w}\sjump{\nabla v}\,ds,
\end{align}
where $\eta>0$ is a real number. Define the following norm for
$v\in H^s(\Omega, \cT_h)$ for $s\geq 2$ :
\begin{equation*}
\|v\|_h^2=\left(\sum_{T\in\cT_h} \|D^2 v\|_{L_2(T)}^2
+\sum_{e\in \cE_h^i}\int_e\frac{\eta}{h_e}\sjump{\nabla v}^2 \,ds\right).
\end{equation*}

We refer to \cite{BSung:2005:DG4,BNeilan:2010:C0IPSingular} for a
proof of the following lemma.
\begin{lemma}\label{lem:Coercivity}
It holds that
\begin{equation*}
 a_h(w,v)\leq \tilde C \|w\|_h\|v\|_h \quad \forall w,\, v \in V_h.
\end{equation*}
For sufficiently large $\eta$, it holds that
\begin{equation*}
\tilde c \|v\|_h^2 \leq a_h(v,v) \quad \forall v\in V_h.
\end{equation*}
\end{lemma}

\subsection{Discrete Boundary Control Problem}
The model we study in this case is the Model Problem 1 described in
Section \ref{sec:Model-Problem}. In this case the space
$V=H^2(\Omega)\cap H^1_0(\Omega)$ and $V_h$ is the one defined in
\eqref{eq:Vh-def}. The space $Q=L_2(\Gamma)$ and $W=L_2(\Omega)$.
Set $Q_{ad}=Q_b$, where $Q_b$ is defined in Section \ref{sec:Model-Problem}.
The discrete control space $Q_h$ is defined by $Q_h=\{p_h\in
L_2(\Gamma): p_h|_{e}\in P_0(e),\; \forall e\in\cE_h^b\}$ and
define the admissible set $Q^h_{ad}=\{p_h\in Q_h: \underline q\leq
p_h\leq \overline q\}$. It is clear that $Q_{ad}^h\subset Q_{ad}$ and $\Pi_hq\in Q_{ad}^h$ whenever $q\in Q_{ad}$.
The operator $B:V\rightarrow Q$ is nothing but the piece-wise ($\Gamma_i$-wise) normal derivative on $\Gamma$ and
$B_h :V_h\rightarrow Q_h$ is
defined by the piecewise (edge-wise) normal derivative, ie.,
$B_hv|_e=\left(\partial v_T/\partial n\right)|_e$, where $v_T=v|_T$ and $T$ be the triangle having the edge $e$ on its boundary. We now verify the assumption
\eqref{eq:Poincare-type} and \eqref{eq:trace-QV}. The inequality
\eqref{eq:Poincare-type} follows from the results on Poincar\'e
type inequalities in \cite{BWZ:2004:Poincare2}. The estimate in
\eqref{eq:trace-QV} follows form the well known trace inequality
on $H^2(\Omega)$ and the properties of enriching function $E_h$ as
follows: Let $v\in V$ and $v_h\in V_h$. Then
\begin{align*}
\sum_{e\in \cE_h^b} \|\partial (v-v_h)/\partial n\|_{0,e}^2 &\leq
2 \sum_{e\in \cE_h^b} \left(\|\partial (v-E_hv_h)/\partial
n\|_{0,e}^2+\|\partial
(v_h-E_hv_h)/\partial n\|_{0,e}^2\right)\\
&=2 \sum_{1\leq i\leq m} \|\partial (v-E_hv_h)/\partial
n\|_{0,\Gamma_i}^2+2 \sum_{e\in \cE_h^b}\|\partial
(v_h-E_hv_h)/\partial n\|_{0,e}^2
\end{align*}
Since $(v-E_hv_h)\in H^2(\Omega)\cap H^1_0(\Omega)$, the trace
inequality in Lemma \ref{lem:trace} implies that
\begin{align*}
\sum_{e\in \cE_h^b} \|\partial (v-v_h)/\partial n\|_{0,e}^2 \leq C
\|v-E_hv_h \|_{2,\Omega}^2+2 \sum_{e\in \cE_h^b} \|\partial
(v_h-E_hv_h)/\partial n\|_{0,e}^2,
\end{align*}
then the triangle inequality yields
\begin{align*}
\sum_{e\in \cE_h^b} \|\partial (v-v_h)/\partial n\|_{0,e}^2 \leq C
\left(\|v-v_h\|_h^2+\sum_{e\in \cE_h^b} \|\partial
(v_h-E_hv_h)/\partial n\|_{0,e}^2\right).
\end{align*}
Now the trace-inverse inequality in Lemma \ref{lem:inverse} on
discrete spaces and Lemma \ref{lem:EnrichApprx} completes the
proof of \eqref{eq:trace-QV}.

The abstract error estimates in Theorem \ref{thm:Err-q} and Theorem
\ref{thm:Err-u} are valid to the model problem under the
discussion.

The error analysis in \cite{BNeilan:2010:C0IPSingular} delivers
the following error estimates for the projections $P_hu$ and
$P_h\phi$:
\begin{align*}
\|u-P_hu\|_h &\leq C\left( \inf_{v_h\in V_h}\|u-v_h\|_h+
h\|f\|_{-1,\Omega}+h^{1/2} \inf_{p_h\in
Q_h}\|q-p_h\|_{0,\Gamma}\right),\\
\|\phi-P_h\phi\|_h &\leq C\left( \inf_{v_h \in V_h}\|\phi-v_h\|_h+
h^2\|u-u_d\|_{0,\Omega}\right).
\end{align*}
Using these estimates, Theorem \ref{thm:Err-q} and Theorem \ref{thm:Err-u}, we obtain the following error estimate:
\begin{align*}
\|q-q_h\|_{0,\Gamma}&+\|u-u_h\|_h+\|\phi-\phi_h\|_h \\
& \leq
C\left( \inf_{v_h \in V_h}\|\phi-v_h\|_h+\inf_{v_h \in
V_h}\|u-v_h\|_h+ h^2\|u-u_d\|_{0,\Omega} +
h\|f\|_{-1,\Omega}\right.\\&\quad\left.+\|\partial \phi/\partial
n-\Pi_h(\partial \phi/\partial
n)\|_{0,\Gamma}+\|q-\Pi_hq\|_{0,\Gamma}+h^{1/2} \inf_{p_h\in
Q_h}\|q-p_h\|_{0,\Gamma}\right).
\end{align*}
\par
\noindent At this moment we can apply the elliptic regularity to
derive concrete error estimates. Note that by the well posed-ness
of the problem, $u,\phi\in H^2(\Omega)$ and $q\in L_2(\Gamma)$.
Then the optimality condition \eqref{eq:MP-q1} implies that

$$q=\Pi_{[\underline q,\,\overline
q]}\left(-\frac{1}{\alpha}\frac{\partial \phi}{\partial
n}\right),$$ where $\Pi_{[a,b]}g(x)$ is defined by
$$\Pi_{[a,b]}g(x)=\min\{b,\max\{a,g(x)\}\}.$$
The elliptic regularity on polygonal domains
\cite{Blum:1980:plates,Grisvard:1992:Singularities} implies that $\phi \in
H^{2+s}(\Omega)$ for some $s\in (0,1]$ which depends on the
interior angles of the domain $\Omega$. Then
\begin{align*}
q|_{\Gamma_i}\in H^{1/2+s}(\Gamma_i)\quad \text{ for all } 1\leq
i\leq m,
\end{align*}
since $(\partial \phi/\partial n)|_{\Gamma_i}\in
H^{1/2+s}(\Gamma_i)$ for $1\leq i\leq m$. Using this, we also get
that $u\in H^{2+s}(\Omega)$.

\smallskip
\noindent Thus we have proved the following theorem:

\begin{theorem}
Let $s\in (0,1]$ be the elliptic regularity index. Then there
holds
\begin{align*}
\|q-q_h\|_{0,\Gamma}+\|u-u_h\|_h+\|\phi-\phi_h\|_h
& \leq
C h^s \left(\|u\|_{2+s,\Omega}+\|\phi\|_{2+s,\Omega}+ h\|u-u_d\|_{0,\Omega}\right)\\
&\quad + C h^s \left(h^{1-s}\|f\|_{-1,\Omega}+h^{1/2}
\sum_{i=1}^m\|q\|_{s, \Gamma_i}\right).
\end{align*}
\end{theorem}
\par
\noindent Define the estimators,
\begin{align*}
\eta_u^2&=\sum_{T\in\cT_h}h_T^2\|f\|_{0,T}^2 +\sum_{e\in\cE_h^i}\int_e\left(h_e\sjump{\partial^2 u_h/\partial n^2}^2+h_e^{-1}\sjump{\nabla u_h}^2\right)ds\\
&\quad +\sum_{e\in\cE_h^b}\int_e h_e\left(\partial^2 u_h/\partial n^2-q_h\right)^2 ds,\\
\text{and}&\\
\eta_\phi^2 &=\sum_{T\in\cT_h}h_T^2\|u_h-u_d\|_{0,T}^2 +\sum_{e\in\cE_h^i}\int_e\left(h_e\sjump{\partial^2 \phi_h/\partial n^2}^2+h_e^{-1}\sjump{\nabla \phi_h}^2\right)ds\\
&\quad +\sum_{e\in\cE_h^b}\int_e h_e\left(\partial^2 \phi_h/\partial n^2\right)^2 ds.
\end{align*}
Again the error analysis in \cite{BNeilan:2010:C0IPSingular} and
the a posteriori error analysis in  \cite{BGS:2010:AC0IP} conclude
the following error estimates:
\begin{align*}
\|Ru-u_h\|_h & \leq C \eta_u, \\
\|R\phi-\phi_h \|_h& \leq C \eta_\phi.
\end{align*}

The following theorem is the consequence of the above two estimates, Theorem  \ref{thm:Err-up-Apost} and Theorem \ref{thm:Apost-up}:
\begin{theorem} There hold
\begin{align*}
\|q-q_h\|_{0,\Gamma}+\|u-u_h\|_h+\|\phi-\phi_h\|_h &\leq C
\left(\eta_u+\eta_\phi\right)\\
&\quad  + C \|(\partial\phi_h/\partial n+\alpha
q_h)-\Pi_h(\partial\phi_h/\partial n+\alpha q_h)\|_{0,\Gamma}.
\end{align*}
\end{theorem}
\subsection{Discrete Distributed Control Problem}
The model problem in this subsection is the Model Problem 2
introduced in Section \ref{sec:Model-Problem}. Set
$V=H^2(\Omega)\cap H^1_0(\Omega)$, $W=L_2(\Omega)$ and
$Q=L_2(\Omega)$. The set $Q_{ad}=Q_d$, where $Q_d$ is defined in
Section \ref{sec:Model-Problem}. The discrete set $V_h$ is the
same as in \eqref{eq:Vh-def}. Set $Q_{ad}=Q_b$, where $Q_b$ is
defined in Section \ref{sec:Model-Problem}. Define the discrete
space $Q_h=\{p_h\in L_2(\Omega): p_h|_{T}\in P_0(T),\; \forall
T\in\cT_h\}$ and the admissible set $Q^h_{ad}=\{p_h\in Q_h:
\underline q\leq p_h\leq \overline q\}$. It is trivial to check
that $Q_{ad}^h\subset Q_{ad}$ and $\Pi_hq\in Q_{ad}^h$ for $q\in
Q_{ad}$. The operator $B:V\rightarrow Q$ and $B_h :V_h\rightarrow
Q_h$ are inclusion (identity) maps. The assumptions
\eqref{eq:Poincare-type} and \eqref{eq:trace-QV} are the
Poincar\'e type inequalities derived in \cite{BWZ:2004:Poincare2}.

The error analysis in \cite{BNeilan:2010:C0IPSingular} implies
the following error estimates for the projections $P_hu$ and $P_h\phi$:
\begin{align*}
\|u-P_hu\|_h &\leq C\left( \inf_{v_h\in V_h}\|u-v_h\|_h+
h\|f\|_{-1,\Omega}+h^{2} \inf_{p_h\in
Q_h}\|q-p_h\|_{0,\Omega}\right),\\
\|\phi-P_h\phi\|_h &\leq C\left( \inf_{v_h \in V_h}\|\phi-v_h\|_h+
h^2\|u-u_d\|_{0,\Omega}\right).
\end{align*}
Using Theorem \ref{thm:Err-q}, Theorem \ref{thm:Err-u} and above estimates, we find
\begin{align*}
\|q-q_h\|_{0,\Omega}&+\|u-u_h\|_h+\|\phi-\phi_h\|_h \\
& \leq
C\left( \inf_{v_h \in V_h}\|\phi-v_h\|_h+\inf_{v_h \in
V_h}\|u-v_h\|_h+ h^2\|u-u_d\|_{0,\Omega} +
h\|f\|_{-1,\Omega}\right.\\&\quad\left.+\|\phi-\Pi_h\phi\|_{0,\Omega}+\|q-\Pi_hq\|_{0,\Omega}+h^{2} \inf_{p_h\in
Q_h}\|q-p_h\|_{0,\Omega}\right).
\end{align*}

\par
\noindent We invoke the elliptic regularity now to derive the
concrete error estimates. Note that
$$q=\Pi_{[\underline q,\,\overline
q]}\left(-\frac{1}{\alpha}\phi\right).$$
By the elliptic regularity \cite{Grisvard:1992:Singularities},
there is some $s\in (0,1]$ which depends on the interior angles of
the domain  $\Omega$ such that $u,\,\phi \in H^{2+s}(\Omega)$ and
hence $q\in W^{1,\infty}(\Omega)$.

Thus we deduce the following theorem as in the case of boundary control problem.

\begin{theorem}
Let $s\in (0,1]$ be the elliptic regularity index. Then there
holds
\begin{align*}
\|q-q_h\|_{0,\Omega}+\|u-u_h\|_h+\|\phi-\phi_h\|_h
& \leq
C h^s \left(\|u\|_{2+s,\Omega}+\|\phi\|_{2+s,\Omega}+ \|u-u_d\|_{0,\Omega}\right)\\
&\quad + C h \left( \|f\|_{-1,\Omega}+h^2 \|q\|_{1,\Omega}\right).
\end{align*}
\end{theorem}
Define the estimators,
\begin{align*}
\eta_u^2&=\sum_{T\in\cT_h}h_T^2\|f+q_h\|_{0,T}^2 +\sum_{e\in\cE_h^i}\int_e\left(h_e\sjump{\partial^2 u_h/\partial n^2}^2+h_e^{-1}\sjump{\nabla u_h}^2\right)ds\\
&\quad +\sum_{e\in\cE_h^b}\int_e h_e\left(\partial^2 u_h/\partial n^2\right)^2 ds,\\
\text{and}&\\
\eta_\phi^2 &=\sum_{T\in\cT_h}h_T^2\|u_h-u_d\|_{0,T}^2 +\sum_{e\in\cE_h^i}\int_e\left(h_e\sjump{\partial^2 \phi_h/\partial n^2}^2+h_e^{-1}\sjump{\nabla \phi_h}^2\right)ds\\
&\quad +\sum_{e\in\cE_h^b}\int_e h_e\left(\partial^2 \phi_h/\partial n^2\right)^2 ds.
\end{align*}

As in the case of boundary control problem, the following theorem
on a posteriori error estimates is a consequence of the results in
\cite{BNeilan:2010:C0IPSingular,BGS:2010:AC0IP}, Theorem
\ref{thm:Err-up-Apost} and Theorem \ref{thm:Apost-up}:
\begin{theorem} There hold
\begin{align*}
\|q-q_h\|_{0,\Omega}+\|u-u_h\|_h+\|\phi-\phi_h\|_h\leq C \left(\eta_u+\eta_\phi+
\|(\phi_h+\alpha q_h)-\Pi_h(\phi_h+\alpha q_h)\|_{0,\Omega}\right).
\end{align*}
\end{theorem}

\par
\noindent {\bf Discussion on the efficiency estimates:}
From the equations \eqref{eq:Ru}-\eqref{eq:Rp}, the continuity of $B:V\rightarrow Q$ and the continuous imbedding of $W\subset V$,
we find
\begin{align*}
\|u-Ru\|_V &\leq C \|q-q_h\|_Q,\\
\|\phi-R\phi\|_V&\leq C \|q-q_h\|_Q.
\end{align*}
Then by the triangle inequality,
\begin{align*}
\|Ru-u_h\|_h &\leq C \left(\|Ru-u\|_V +\|u-u_h\|_h\right)\leq C  \left(\|q-q_h\|_Q+\|u-u_h\|_h\right),\\
\|R\phi-\phi_h\|_h&\leq C \left(\|R\phi-\phi\|_V+\|\phi-\phi_h\|_h\right)\leq C \left(\|q-q_h\|_Q+\|\phi-\phi_h\|_h\right).
\end{align*}
Therefore the efficiency of the terms in $\eta_u$ and $\eta_\phi$
follows by the standard bubble function techniques. Then note, for
example, in the case of distributed control that by the triangle
inequality and the stability of the projection $\Pi_h$ that
\begin{align*}
\|(\phi_h+\alpha q_h)-\Pi_h(\phi_h+\alpha q_h)\|_{0,T} &\leq
\|(\phi_h+\alpha q_h)-(\phi+\alpha q)\|_{0,T}+\|(\phi+\alpha
q)-\Pi_h(\phi+\alpha q)\|_{0,T}\\
&\quad +\|\Pi_h(\phi+\alpha q)-\Pi_h(\phi_h+\alpha q_h)\|_{0,T}\\
&\leq C \|(\phi_h+\alpha q_h)-(\phi+\alpha q)\|_{0,T}\\&\quad
+\|(\phi+\alpha q)-\Pi_h(\phi+\alpha q)\|_{0,T}.
\end{align*}
This completes the discussion on the efficiency of the error estimates.

\section{Numerical Experiments}\label{sec:Numerics}
In this section, we present some numerical experiments to
illustrate the theoretical results derived in the article. In all
the examples below, we choose the penalty parameter $\eta =10,$ $\underline{q}=-750$, $\overline{q}=-50$
and $\alpha=10^{-3}$. The discrete solution is computed by using the primal-dual active set
algorithm in \cite{trolzstch:2005:Book}.

{\bf Example 1}. First, we test the a priori error
estimates for a model of distributed optimal control problem with
homogeneous simply supported plate boundary conditions. The
computational domain is chosen to be $\Omega=(0,1)^2$. The data of
the model problem is constructed in such a way that the exact
solution is known. This is done  by choosing the state variable
$u$, the adjoint variable $\phi$ as
\begin{equation*}
u(x,y)=\phi(x,y)=\sin^3(\pi x) \sin^3(\pi y),
\end{equation*}
and the control $q$ as $q(x)= \Pi_{[-750, -50]}\left( -
\frac{1}{\alpha} \phi(x) \right)$.
The source term $f$ and the observation $u_d$ are then computed by
using
\begin{equation*}
f=\Delta^2 u-q, \quad u_d=u-\Delta^2 \phi.
\end{equation*}

\par
\noindent We take a sequence of uniformly refined meshes with mesh
parameter $h$ as it is shown in Table \ref{table:H2Y}.  The exact
errors and orders of convergence have been computed and shown in
the Table \ref{table:H2Y}. The results clearly predicts the linear
rate of convergence as it is expected.

\begin{table}[h!!]
 {\small{\footnotesize
\begin{center}
\begin{tabular}{|c|c|c|c|c|c|c|c|}\hline
 $h$ & $\|u-u_h\|_h $  & order &  $\|\phi-\phi_h\|_h$ & order & $\|q-q_h\|_{0,\Omega}$ & order\\
\hline\\[-12pt]
 1/4   & 11.7524   &  --     &  18.0932  &   --     & 264.3128 &  --  \\
 1/8   & 6.5598    & 0.8412  &  6.5644   &   1.4627 & 58.6117  &   2.1730     \\
 1/16  & 3.3721    & 0.9600  &  3.3808   &   0.9573 &  29.7993 &  0.9759  \\
 1/32  & 1.6701    & 1.0137  &  1.6719   &   1.0158 &  14.6533 &   1.0241   \\
 1/64  & 0.8286    &  1.0112 &  0.8289   &   1.0123 &  7.3130  &   1.0027  \\
 1/128 & 0.4133    &  1.0037 &  0.4133   &   1.0040 & 3.6581   &  0.9994  \\
\hline
\end{tabular}
\end{center}
}}
\par\medskip
\caption{Errors and orders of convergence for Example 1 }
\label{table:H2Y}
\end{table}

Now, we test the performance of the a
posteriori error estimator for the above distributed control problem. Note that the state and adjoint state are smooth but not the control.

\par
\noindent The following algorithm for adaptive refinement has been
used:
\begin{center}
{SOLVE $\rightarrow$ ESTIMATE $\rightarrow$ MARK $\rightarrow$
REFINE }
\end{center}
We compute the discrete solutions and then we compute the
error estimator and mark the elements using the D\"orlfer marking
technique \cite{Dorlfer:1996:Marking} with $\theta=0.3$. We refine
the marked elements using the newest vertex bisection algorithm
and obtain a new mesh. Figure \ref{fig:err1} shows the behavior of
the estimator and the errors $\|u-u_h\|_h, \|\phi-\phi_h\|_h$ and
$\|q-q_h\|_{0,\Omega}$ with the increasing number of degrees of
freedom $N$ (number of unknowns for state variable). We observe
that the estimator is reliable. The errors in state, adjoint state
and control converge at the optimal rate of $1/\sqrt{N}$.  The
efficiency of the estimator is depicted through the efficiency
indices
$\left(\text{estimator}/(\|u-u_h\|_h+\|\phi-\phi_h\|_h+\|q-q_h\|_{0,\Omega})\right)$
in Figure \ref{fig:Eff}. Finally, Figure \ref{fig:MeshL} shows the
adaptive mesh refinement.

\begin{figure}[!ht]
\begin{center}
\includegraphics[width=9cm,height=7cm]{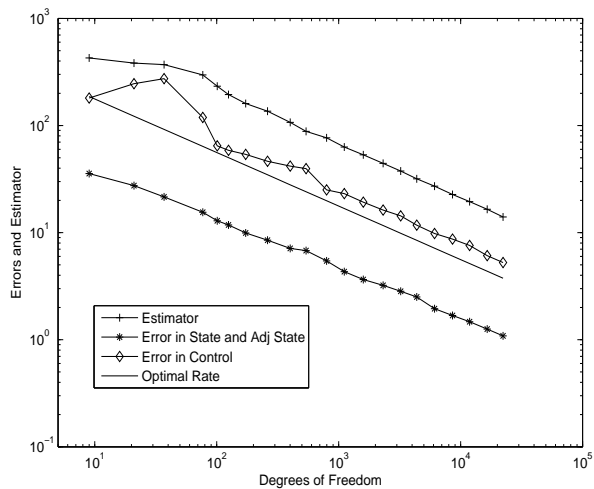}
\caption{Errors and estimator for Example 1 } \label{fig:err1}
\end{center}
\end{figure}
\begin{figure}[!ht]
\begin{center}
\includegraphics[width=9cm,height=7cm]{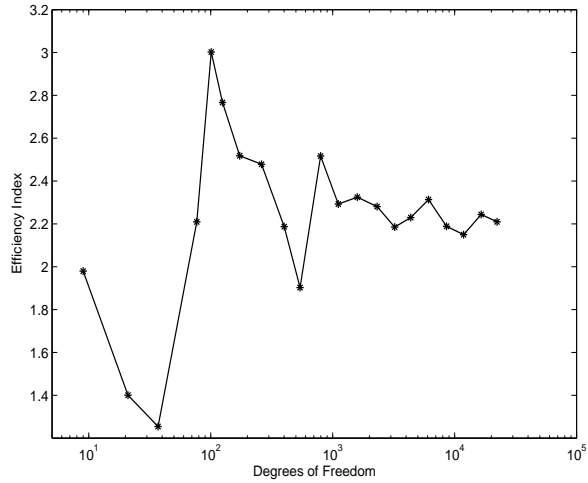}
\caption{Efficiency Index for Example 1 }\label{fig:Eff}
\end{center}
\end{figure}
\begin{figure}[!ht]
\begin{center}
\includegraphics[width=9cm,height=7cm]{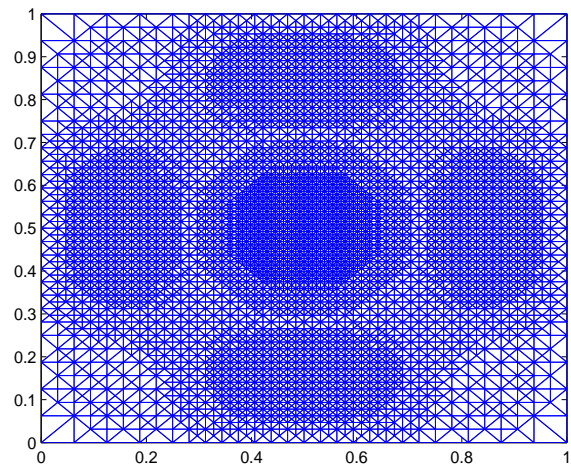}
\caption{Adaptive mesh refinement for Example 1 }\label{fig:MeshL}
\end{center}
\end{figure}


{\bf Example 2.} In this example, we test the performance of the a
posteriori error estimator for a distributed control problem in
the presence of re-entrant corners. The domain $\Omega$ is set to
be $L-$shaped as it is shown in Figure \ref{fig:mesh2}. We set the
source term $f=1$ and the observation  $u_d=1$. In this case since
we do not have exact solutions at hand,  we test the optimal
convergence of the error estimator and its performance in
capturing the re-entrant corner.  The numerical experiment shows
that the error estimator converges optimally (see Figure
\ref{fig:err2}) and refines the mesh locally at the reentrant
corner (see Figure \ref{fig:mesh2}) as it should be expected.

\begin{figure}[!ht]
\begin{center}
\includegraphics[width=9cm,height=7cm]{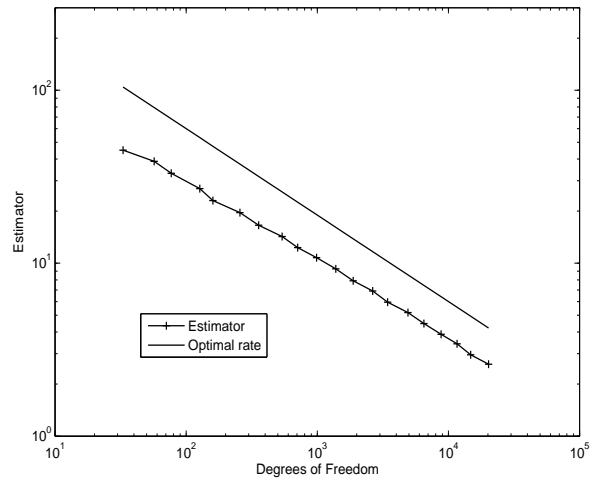}
\caption{Errors and estimator for Example 2 } \label{fig:err2}
\end{center}
\end{figure}
\begin{figure}[!ht]
\begin{center}
\includegraphics[width=9cm,height=7cm]{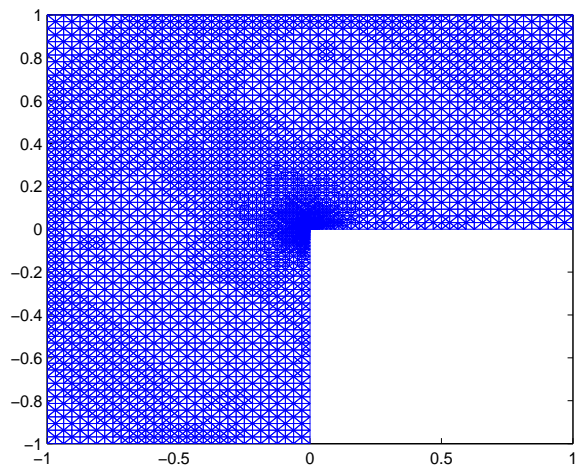}
\caption{Mesh refinement for Example 2 }\label{fig:mesh2}
\end{center}
\end{figure}

\section{Variational discontinuous discretization method}\label{sec:VDM}
In this section, the variational discretization method (control
being not discretized) introduced in \cite{Hinze:2005:Control}
will be discussed in the context of discontinuous Galerkin methods
and then error estimates will be discussed. For this we use the
notation and the setting in Section \ref{sec:Abstract}. The
variational discontinuous discretization method is defined as to
find $(\tilde u_h,\tilde \phi_h,\tilde q)\in V_h\times V_h\times
Q_{ad}$ such that
\begin{align}
a_h(\tilde u_h,v_h)&=f(v_h)+\langle \tilde q,B_hv_h\rangle \quad \forall v_h\in V_h,\quad\label{N:eq:ADP-u}\\
a_h(v_h,\tilde\phi_h)&=(\tilde u_h-u_d,v_h)\quad \forall v_h\in V_h,\quad\label{N:eq:ADP-p}\\
\langle B_h\tilde\phi_h+\alpha \tilde q, p- \tilde q\rangle &\geq
0\quad \forall p\in Q_{ad}, \label{N:eq:ADP-q}
\end{align}

\par
\smallskip
\noindent The following lemma is proved in the same lines as that
of Lemma \ref{lem:Err-q-abstract}:

\begin{lemma}\label{N:lem:Err-q-abstract} There hold
\begin{align*}
\langle B_h(\tilde \phi_h-P_h\phi), q-\tilde q\rangle &\geq \alpha
\|q-\tilde q\|_Q^2+\langle B_h(\phi-P_h\phi), q-\tilde q\rangle.
\end{align*}
\end{lemma}
\begin{proof}
From \eqref{N:eq:ADP-q} and \eqref{eq:AMP-q}, we have
\begin{align*}
\langle B_h\tilde \phi_h+\alpha \tilde q, q-\tilde q\rangle
&\geq 0,\\
-\langle B_h\phi+\alpha q, q- \tilde q\rangle &\geq 0,
\end{align*}
and by adding them, we find
\begin{align*}
\langle B_h(\tilde \phi_h-\phi)+\alpha (\tilde q-q), q-\tilde
q\rangle &\geq 0.
\end{align*}
This implies
\begin{align*}
\langle B_h(\tilde \phi_h-P_h\phi), q-\tilde q\rangle &\geq \alpha
\|q-\tilde q\|_Q^2+\langle B_h(\phi-P_h\phi), q-\tilde q\rangle,
\end{align*}
and hence the proof is completed.
\end{proof}

The error estimate for the control and the state are derived in
the following theorem.

\begin{theorem}\label{N:thm:Err-q}
There hold
\begin{align*}
\|q-\tilde q\|_Q+\|u-\tilde u_h\|_W \leq C
\left(\|B_h(\phi-P_h\phi)\|_Q+\|u-P_hu\|_W\right).
\end{align*}
\end{theorem}
\begin{proof}
From \eqref{N:eq:ADP-u}-\eqref{N:eq:ADP-p} and the definition of
$P_h$, we have
\begin{align}
a_h(P_hu-\tilde u_h,v_h)&= \langle q-\tilde q,B_hv_h\rangle\quad \forall v_h\in V_h,\label{N:eq:E11}\\
a_h(v_h, P_h\phi-\tilde \phi_h)&=(u-\tilde u_h,v_h)\quad \forall
v_h\in V_h.\label{N:eq:E12}
\end{align}
Take $v_h=P_hu-\tilde u_h$ in \eqref{N:eq:E12},
$v_h=P_h\phi-\tilde \phi_h$ in \eqref{N:eq:E11} and subtract the
resulting equations to find
\begin{align*}
\langle q-\tilde q, B_h(P_h\phi-\tilde \phi_h)\rangle-(u-\tilde
u_h,P_hu-\tilde u_h)=0.
\end{align*}
This implies
\begin{align*}
\langle q-\tilde q,B_h(\tilde \phi_h-P_h\phi)\rangle+\|P_hu-\tilde
u_h\|_W^2=(u-P_hu,\tilde u_h-P_hu).
\end{align*}
Using Lemma \ref{N:lem:Err-q-abstract}, we find
\begin{align*}
\alpha \|q-\tilde q\|_Q^2+\|P_hu-\tilde u_h\|_W^2 &\leq
-\langle B_h(\phi-P_h\phi), q-\tilde q\rangle + (u-P_hu,\tilde u_h-P_hu)\notag \\
&\leq C \left(\|B_h(\phi-P_h\phi)\|_Q^2+\|u-P_hu\|_W^2\right)\notag\\
&\quad + \frac{1}{2}\|\tilde
u_h-P_hu\|_W^2+\frac{\alpha}{2}\|q-\tilde q\|_Q^2.
\end{align*}
This completes the proof.
\end{proof}

The energy norm error estimates for the state and the adjoint
state are proved in the following:

\begin{theorem}\label{N:thm:Err-u} There hold
\begin{align*}
\|\phi-\tilde \phi_h\|_h &\leq C\left( \|\phi-P_h\phi\|_h +
\|B_h(\phi-P_h\phi)\|_Q+\|u-P_hu\|_W\right),\\
\|u-u_h\|_h & \leq
C\left(\|u-P_hu\|_h+\|B_h(\phi-P_h\phi)\|_Q+\|u-P_hu\|_W\right).
\end{align*}
\end{theorem}

\begin{proof}
The estimate in Theorem \ref{N:thm:Err-q}, the continuity of $a_h$
on $V_h$, the error equation \eqref{N:eq:E12}, the assumption
(P-T) and the triangle inequality imply
\begin{align*}
\|\phi-\tilde\phi_h\|_h&\leq \|\phi-P_h\phi\|_h+\|P_h\phi-\tilde \phi_h\|\\
&\leq \|\phi-P_h\phi\|_h+ C\|u-\tilde u_h\|_W\\
&\leq C\left( \|\phi-P_h\phi\|_h +
\|B_h(\phi-P_h\phi)\|_Q+\|u-P_hu\|_W\right).
\end{align*}
The error equation \eqref{N:eq:E11} and the assumption
\eqref{eq:trace-QV} imply
\begin{align*}
a_h(P_hu-\tilde u_h,P_hu-\tilde u_h)&=\langle q-\tilde q,B_h(P_hu-\tilde u_h)\rangle\\
&\leq C\|q-\tilde q\|_Q \,\|P_h u-\tilde u_h\|_h.
\end{align*}
The rest of the proof follows from Theorem \ref{N:thm:Err-q}.
\end{proof}

Similarly, we can find a posteriori error estimates in the same
lines. In this case the a posteriori error estimator does not
involve the term $\|(B_h\tilde \phi_h+\alpha \tilde
q)-\Pi_h(B_h\tilde \phi_h+\alpha \tilde q )\|_Q$.

\section{Conclusions}\label{sec:conlusions}
We have developed a framework for the error analysis of
discontinuous finite element methods for elliptic optimal control
problems with control constraints. The abstract analysis provides
best approximation results which will be useful in convergence of
adaptive methods and delivers a reliable and efficient a
posteriori error estimator. The results are applicable to a
variety of discontinuous Galerkin methods (including classical
nonconforming methods) applied to elliptic optimal control
problems (distributed and Neumann) with constraints on control.
Applications to $C^0$ interior penalty methods for optimal control
problems governed by the biharmonic equation with simply supported
boundary conditions are established.  Numerical experiments
illustrate the theoretical findings. Variational discretization
method is discussed in the context of discontinuous Galerkin
methods and corresponding error estimates are derived.

\par
\noindent {\bf Acknowledgements:} The first author would like to
acknowledge the support from NBHM-India, the second author would
like to acknowledge the support from DST fast track project and
all the authors would like to thank the UGC Center for Advanced
Study.

\end{document}